\pdfoutput=1
\RequirePackage{ifpdf}
\ifpdf 
\documentclass[pdftex]{sigma}
\else
\documentclass{sigma}
\fi

\numberwithin{equation}{section}

\newtheorem{Theorem}{Theorem}[section]
\newtheorem{Corollary}[Theorem]{Corollary}

\newcommand{\ad}{\operatorname{ad}}

\begin{document}
\allowdisplaybreaks

\newcommand{\arXivNumber}{2004.12875}

\renewcommand{\thefootnote}{}

\renewcommand{\PaperNumber}{118}

\FirstPageHeading

\ShortArticleName{New Pieri Type Formulas for Jack Polynomials and their Applications}

\ArticleName{New Pieri Type Formulas for Jack Polynomials and\\ their Applications to Interpolation Jack Polynomials\footnote{This paper is a~contribution to the Special Issue on Elliptic Integrable Systems, Special Functions and Quantum Field Theory. The full collection is available at \href{https://www.emis.de/journals/SIGMA/elliptic-integrable-systems.html}{https://www.emis.de/journals/SIGMA/elliptic-integrable-systems.html}}}

\Author{Genki SHIBUKAWA}

\AuthorNameForHeading{G.~Shibukawa}

\Address{Department of Mathematics, Kobe University, Rokko, Kobe 657-8501, Japan}
\Email{\href{mailto:g-shibukawa@math.kobe-u.ac.jp}{g-shibukawa@math.kobe-u.ac.jp}}

\ArticleDates{Received April 29, 2020, in final form November 14, 2020; Published online November 21, 2020}

\Abstract{We present new Pieri type formulas for Jack polynomials. As an application, we give a new derivation of higher order difference equations for interpolation Jack polynomials originally found by Knop and Sahi. We also propose Pieri formulas for interpolation Jack polynomials and intertwining relations for a kernel function for Jack polynomials.}


\Classification{05E05; 33C67; 43A90}

\renewcommand{\thefootnote}{\arabic{footnote}}
\setcounter{footnote}{0}

\begin{flushright}
\it Dedicated to M.~Noumi for his 65th birthday.
\end{flushright}

\section{Introduction}
Given a positive integer $r$ and a non-zero complex parameter $d$,
we define a second-order dif\-fe\-ren\-tial operator in the variables $\mathbf{z}=(z_{1},\ldots,z_{r})$ by
\begin{gather*}
D(\mathbf{z}) := \sum_{j=1}^{r} z_{j}^{2}\partial_{z_{j}}^{2}
 +d \sum_{1\leq j\not= l \leq r} \frac{z_{j}^{2}}{z_{j}-z_{l}} \partial_{z_{j}},
\end{gather*}
where $\partial_{z_{j}}:=\frac{\partial }{\partial z_{j}}$ for $j=1,\ldots,r$.
The {\it{Jack polynomials}} are a family of homogeneous symmetric polynomials $P_{\mathbf{m}}\left(\mathbf{z};\frac{d}{2}\right)$, indexed by partitions $\mathbf{m}=(m_{1},\ldots,m_{r})$, which are characterized by the differential equation
\begin{gather*}
D(\mathbf{z})P_{\mathbf{m}}\bigg(\mathbf{z};\frac{d}{2}\bigg) =
 P_{\mathbf{m}}\bigg(\mathbf{z};\frac{d}{2}\bigg) \sum_{j=1}^{r}m_{j}(m_{j}-1+d(r-j)),
\end{gather*}
together with the triangularity
\begin{gather*}
P_{\mathbf{m}}\bigg(\mathbf{z};\frac{d}{2}\bigg) =
\sum_{\mathbf{k}\leq \mathbf{m}} c_{\mathbf{m}\mathbf{k}}m_{\mathbf{k}}(\mathbf{z}), \qquad
c_{\mathbf{m}\mathbf{k}} \in \mathbb{Q}(d), \qquad c_{\mathbf{m}\mathbf{m}}=1
\end{gather*}
with respect to the dominance order on partitions
$\mathbf{m}=(m_1,\dots,m_r)$ of a fixed weight $|\mathbf{m}|:=m_1+\dots+m_r$
\begin{gather*}
 \mathbf{k} \leq \mathbf{m} \quad \Longleftrightarrow \quad
 \sum_{l=1}^{i}k_{l}\leq \sum_{l=1}^{i}m_{l}, \qquad i=1,\ldots, r,
\end{gather*}
and the monomial symmetric polynomials
$m_{\mathbf{m}}(\mathbf{z}) :=
 \sum_{\mathbf{n} \in \mathfrak{S}_{r}.\mathbf{m}}z^{\mathbf{n}}$,
where $\mathfrak{S}_{r}$ denotes the symmetric group of degree $r$.
Since Jack polynomials $P_{\mathbf{m}}\left(\mathbf{z};\frac{d}{2}\right)$ are orthogonal polynomials associated with the root system $A_{r-1}$ and eigenstates of the Calogero--Surtherland model, they play important roles in various fields of mathematics and physics.
For their fundamental properties and applications, see \cite{M2, St, VK}.

The {\it{interpolation Jack polynomials}} (or {\it{shifted Jack polynomials}}) are also a family of symmetric polynomials indexed by partitions $\mathbf{m}$, which we denote by $P_{\mathbf{m}}^{\mathrm{ip}}\left(\mathbf{z};\frac{d}{2}\right)$ according to the notation of Koornwinder \cite{Ko}.
They are uniquely defined by the following two conditions:
\begin{gather*}
P_{\mathbf{k}}^{\mathrm{ip}}\bigg(\mathbf{m}+\frac{d}{2}\boldsymbol{\delta } ;\frac{d}{2}\bigg)=0, \qquad \text{unless $\mathbf{k} \subseteq \mathbf{m} \in \mathcal{P}$},
\\
P_{\mathbf{m}}^{\mathrm{ip}}\bigg(\mathbf{z};\frac{d}{2}\bigg)
 = P_{\mathbf{m}}\bigg(\mathbf{z};\frac{d}{2}\bigg)
 +\text{(lower degree terms)},
\end{gather*}
where $\boldsymbol{\delta }$ denotes the staircase partition $(r-1,r-2,\ldots ,2,1,0)$ and $\mathbf{k} \subseteq \mathbf{m}$ stands for the inclusion partial order defined by
\begin{gather*}
\mathbf{k}\subseteq \mathbf{m} \quad \Longleftrightarrow \quad k_{i}\leq m_{i}, \qquad i=1,\ldots, r.
\end{gather*}
These polynomials $P_{\mathbf{m}}^{\mathrm{ip}}\left(\mathbf{z};\frac{d}{2}\right)$ were introduced by Sahi~\cite{Sa1}, Knop--Sahi \cite{KS} and Okounkov--Olshanski \cite{OO1} as a continuous deformation of shifted Schur polynomials $P_{\mathbf{m}}^{\mathrm{ip}}\left(\mathbf{z};1\right)$ \cite{M1,OO2}.
The interpolation Jack polynomials appear as a multivariate analogue of the falling factorials
\begin{gather*}
P_{m}^{\mathrm{ip}}(z) :=
 \begin{cases}
 z(z-1)\cdots (z-m+1), & m\not=0, \\
 1, & m=0,
 \end{cases}
\end{gather*}
in explicit formulas of binomial type for multivariate hypergeometric functions (see \cite{Ko, O1,O2}).

We review some fundamental results on Jack and interpolation Jack polynomials relevant to~the main results of this paper.

{\bf{1) Sekiguchi operators}} (Debiard \cite{D}, Macdonald \cite{M2}, Sekiguchi \cite{Se}).
The Jack polynomials are joint eigenfunctions of a commuting family of differential operators ({\it{Sekiguchi operators}}).
We define Sekiguchi operators $H_{r,p}^{(d)}(\mathbf{z})$ $(p=0,1,\ldots,r)$ and their generating function $S_{r}^{(d)}(u;\mathbf{z})$ by
\begin{gather*}
H_{r,p}^{(d)}(\mathbf{z}) := \sum_{l=0}^{p}\bigg(\frac{2}{d}\bigg)^{p-l}
\sum_{\substack{I \subseteq [r],\\ |I|=l}}
\bigg(\frac{1}{\Delta (\mathbf{z})}\bigg(\prod_{i \in I}
 z_{i}\partial_{z_{i}}\bigg)\Delta (\mathbf{z})\bigg)
 \sum_{\substack{J \subseteq [r]\setminus I,\\ |J|=p-l}}
 \bigg(\prod_{j \in J} z_{j}\partial_{z_{j}}\bigg),
 \\
S_{r}^{(d)}(u;\mathbf{z}) :=\sum_{p=0}^{r}H_{r,p}^{(d)}(\mathbf{z})u^{r-p}
\end{gather*}
respectively, where $[r]:=\{1,2,\ldots ,r\}$ and $\Delta (\mathbf{z}):=\prod_{1\leq i<j\leq r}(z_{i}-z_{j})$.
Then we have
\begin{gather}
H_{r,p}^{(d)}(\mathbf{z})P_{\mathbf{m}}\bigg(\mathbf{z};\frac{d}{2}\bigg)
 = P_{\mathbf{m}}\bigg(\mathbf{z};\frac{d}{2}\bigg)
 e_{r,p}\bigg(\mathbf{m}+\frac{d}{2}\boldsymbol{\delta } \bigg), \nonumber
 \\
\label{eq:Sekiguchi gen}
S_{r}^{(d)}(u;\mathbf{z})P_{\mathbf{m}}\bigg(\mathbf{z};\frac{d}{2}\bigg)
 = P_{\mathbf{m}}\bigg(\mathbf{z};\frac{d}{2}\bigg)
 I_{r}^{(d)}(u;\mathbf{m}),
\end{gather}
where
\begin{gather*}
e_{r,k}(\mathbf{z}):=
 \!\!\!\!\sum_{1\leq i_{1}<\cdots <i_{k}\leq r}\!\!\!\!
 z_{i_{1}}\cdots z_{i_{k}}, \qquad k=1,\ldots ,r, \qquad
 e_{r,0}(\mathbf{z}):=1,
 \\
I_{r}^{(d)}(u;\mathbf{m}) := \prod_{k=1}^{r}\bigg(u+r-k+\frac{2}{d}m_{k}\bigg)
 =\bigg(\frac{2}{d}\bigg)^{r} \prod_{k=1}^{r}\bigg(m_{k}+\frac{d}{2}(u+r-k)\bigg).
\end{gather*}

{\bf{2) Pieri type formulas for Jack polynomials}} (Lassalle \cite{L}, Stanley \cite{St}).
We introduce two kinds of normalization for Jack polynomials,
\begin{gather*}
\Phi _{\mathbf{m}}^{(d)}(\mathbf{z}) :=
\frac{P_{\mathbf{m}}\left(\mathbf{z};\frac{d}{2}\right)}{P_{\mathbf{m}}\left(\mathbf{1};\frac{d}{2}\right)}, \\
\Psi _{\mathbf{m}}^{(d)}(\mathbf{z}) :=
 \frac{P_{\mathbf{m}}\left(\mathbf{z};\frac{d}{2}\right)}{P_{\mathbf{m}}^{\mathrm{ip}}
 \left(\mathbf{m}+\frac{d}{2}\boldsymbol{\delta };\frac{d}{2}\right)} =
 \frac{P_{\mathbf{m}}\left(\mathbf{1};\frac{d}{2}\right)}{P_{\mathbf{m}}^{\mathrm{ip}}
 \left(\mathbf{m}+\frac{d}{2}\boldsymbol{\delta } ;\frac{d}{2}\right)}
 \Phi _{\mathbf{m}}^{(d)}(\mathbf{z}),
\end{gather*}
where $\mathbf{1}:=(1,\ldots,1)$.
For the explicit formulas of $P_{\mathbf{m}}\left(\mathbf{1};\frac{d}{2}\right)$ and $P_{\mathbf{m}}^{\mathrm{ip}}\left(\mathbf{m}+\frac{d}{2}\boldsymbol{\delta } ;\frac{d}{2}\right)$, we refer to~\cite{M2} and \cite{Ko}.
We also use the notations
\begin{gather*}
\mathcal{P}:= \{\mathbf{m}=(m_{1},\ldots,m_{r}) \in \mathbb{Z}^{r}
\mid m_{1}\geq \cdots \geq m_{r} \geq 0\},
\\
|\mathbf{z}|:= \sum_{j=1}^{r}z_{j}, \qquad
|\partial_{\mathbf{z}}| := \sum_{j=1}^{r}\partial_{z_{j}}, \qquad
A_{\pm ,i}^{(d)}(\mathbf{x})
 :=\!
 \prod_{1\leq k\not=i\leq r}\!\frac{x_{i}-x_{k}-\frac{d}{2}(i-k)\pm \frac{d}{2}}{x_{i}-x_{k}-\frac{d}{2}(i-k)},
 \\
\epsilon_{i} := (0,\ldots,0,\stackrel{i}{\stackrel{\vee}{1}},0,\ldots,0) \in \mathbb {Z}^{r}.
\end{gather*}
Then the Jack polynomials satisfy two Pieri type formulas\footnote{Although these naming for (\ref{eq:PhiPieri-}) and (\ref{eq:PsiPieri-}) may not be appropriate, we have not been able to give an appropriate name for these formulas and will call them Pieri type formulas in this paper.} for the differential operator $|\partial_{\mathbf{z}}|$ \mbox{\cite[equation~(14.1)]{L}}
\begin{gather}
\label{eq:PhiPieri-}
|\partial_{\mathbf{z}}|\Phi_{\mathbf{x}}^{(d)}(\mathbf{z})
 = \sum_{i=1}^{r} \Phi_{\mathbf{x}-\epsilon_{i}}^{(d)}(\mathbf{z})
 \bigg(x_{i}+\frac{d}{2}(r-i)\bigg) A_{-,i}^{(d)}(\mathbf{x}),
 \\
\label{eq:PsiPieri-}
|\partial_{\mathbf{z}}|\Psi_{\mathbf{x}}^{(d)}(\mathbf{z})
= \sum_{\substack{1\leq i\leq r, \\ \mathbf{x}-\epsilon_{i} \in \mathcal{P}}}
\Psi_{\mathbf{x}-\epsilon_{i}}^{(d)}(\mathbf{z}) A_{+,i}^{(d)}(\mathbf{x}-\epsilon_{i}),
\end{gather}
and for the multiplication operator $|\mathbf{z}|$ \cite[Theorem~6.1]{St} (see also Corollary~\ref{thm:Jack Pieri} in our article and \cite[equation~(6.24)]{M2})
\begin{gather}
\label{eq:PhiPieri+}
|\mathbf{z}|\Phi_{\mathbf{x}}^{(d)}(\mathbf{z}) = \sum_{i=1}^{r}
 \Phi_{\mathbf{x}+\epsilon_{i}}^{(d)}(\mathbf{z})
 A_{+,i}^{(d)}(\mathbf{x}),
 \\
|\mathbf{z}|\Psi_{\mathbf{x}}^{(d)}(\mathbf{z})=
 \sum_{\substack{1\leq i\leq r, \\ \mathbf{x}+\epsilon_{i} \in \mathcal{P}}}
 \Psi_{\mathbf{x}+\epsilon_{i}}^{(d)}(\mathbf{z})
 \bigg(x_{i}+1+\frac{d}{2}(r-i)\bigg)
 A_{-,i}^{(d)}(\mathbf{x}+\epsilon_{i}).\nonumber
\end{gather}
We remark that if $\mathbf{x}-\epsilon_{i} \not\in \mathcal{P}$ (resp.~$\mathbf{x}+\epsilon_{i} \not\in \mathcal{P}$) then $A_{-,i}^{(d)}(\mathbf{x})=0$ (resp.~$A_{+,i}^{(d)}(\mathbf{x})=0$).

{\bf{3) Binomial formula}} (Knop--Sahi \cite{KS}, Okounkov--Olshanski \cite{OO1}).
For any partition $\mathbf{x}$, Jack polynomials with variables shifted by $\mathbf{1}$ are written in the form
\begin{gather}
\label{eq:binomial}
\Phi _{\mathbf{x}}^{(d)}(\mathbf{1}+\mathbf{z}) = \sum_{\mathbf{k} \subseteq \mathbf{x}}
 \frac{P_{\mathbf{k}}^{\mathrm{ip}}
 \left(\mathbf{x}+\frac{d}{2}\boldsymbol{\delta} ;\frac{d}{2}\right)}
 {P_{\mathbf{k}}\left(\mathbf{1} ;\frac{d}{2}\right)}\Psi_{\mathbf{k}}^{(d)}(\mathbf{z}).
\end{gather}

In this article, we propose new Pieri type formulas for Jack polynomials that generalize (\ref{eq:PhiPieri-}) and (\ref{eq:PsiPieri-}), which we call the {\it{twisted Pieri formulas}}.
They give explicit expressions of the action of the operator
\begin{gather*}
\frac{(\ad{|\partial_{\mathbf{z}}|})^{l}}{l!}S_{r}^{(d)}(u;\mathbf{z})
\end{gather*}
on $\Phi_{\mathbf{x}}^{(d)}(\mathbf{z})$ and $\Psi_{\mathbf{x}}^{(d)}(\mathbf{z})$, where $\ad$ denotes the usual commutator defined by $(\ad{A})(B):=$ \mbox{$AB-BA$}.
\begin{Theorem}[twisted Pieri formulas for Jack polynomials]\label{thm:twisted Pieri}
For $l=0,1,\ldots, r$, we have
\begin{gather}
\label{eq:Phikey lemma}
\bigg[\frac{(\ad{|\partial_{\mathbf{z}}|})^{l}}{l!}S_{r}^{(d)}(u;\mathbf{z})\bigg]
\Phi_{\mathbf{x}}^{(d)}(\mathbf{z})
= \sum_{\substack{J \subseteq [r], \\ |J|=l}}
 \Phi_{\mathbf{x}-\epsilon_{J}}^{(d)}(\mathbf{z})
 I_{J^{c}}^{(d)}(u;\mathbf{x}) A_{-,J}^{(d)}(\mathbf{x})
 \prod_{j \in J}\bigg(x_{j}+\frac{d}{2}(r-j)\bigg),
 \\
\label{eq:Psikey lemma}
 \bigg[\frac{(\ad{|\partial_{\mathbf{z}}|})^{l}}{l!}S_{r}^{(d)}(u;\mathbf{z})\bigg]
 \Psi_{\mathbf{x}}^{(d)}(\mathbf{z}) =\!\!\!\!\!\!
 \sum_{\substack{J \subseteq [r], |J|=l, \\ \mathbf{x}-\epsilon_{J} \in \mathcal{P}}}\!\!\!\!\!\!
 \Psi_{\mathbf{x}-\epsilon_{J}}^{(d)}(\mathbf{z})
 I_{J^{c}}^{(d)}(u;\mathbf{x})
 A_{+,J}^{(d)}(\mathbf{x}-\epsilon_{J}),
\end{gather}
where $J^{c}:=[r]\setminus J$, $\epsilon_{J}:=\sum_{j\in J}\epsilon _{j}$ and
\begin{gather*}
A_{\pm ,J}^{(d)}(\mathbf{x}) :=
 \!\!\prod_{j \in J,\, l \in J^{c}}\!\!\frac{x_{j}-x_{l}-\frac{d}{2}(j-l)\pm \frac{d}{2}}{x_{j}-x_{l}-\frac{d}{2}(j-l)},
 \\
I_{J^{c}}^{(d)}(u;\mathbf{x})
 :=
 \bigg(\frac{2}{d}\bigg)^{r}
 \prod_{l \in J^{c}}
 \bigg(x_{l}+\frac{d}{2}(u+r-l)\bigg).
\end{gather*}
\end{Theorem}
Since the first-order Sekiguchi operator $H_{r,1}^{(d)}(\mathbf{z})$ coincides with the Euler operator $\sum_{i=1}^{r}z_{i}\partial_{z_{i}}$ essentially
\begin{gather*}
H_{r,1}^{(d)}(\mathbf{z}) =
 \frac{2}{d}\sum_{i=1}^{r}z_{i}\partial_{z_{i}} +\frac{r(r-1)}{2},
\end{gather*}
we write down the operator $(\ad{|\partial_{\mathbf{z}}|})H_{r,1}^{(d)}(\mathbf{z})$
\begin{gather*}
(\ad{|\partial_{\mathbf{z}}|})H_{r,1}^{(d)}(\mathbf{z}) =
 (\ad{|\partial_{\mathbf{z}}|})\bigg(\frac{2}{d}\sum_{i=1}^{r}z_{i}\partial_{z_{i}} \bigg)
 = \frac{2}{d}|\partial_{\mathbf{z}}|.
\end{gather*}
Therefore twisted Pieri formulas (\ref{eq:Phikey lemma}) and (\ref{eq:Psikey lemma}) are regarded as a higher order analogue of the above Pieri type formulas (\ref{eq:PhiPieri-}) and (\ref{eq:PsiPieri-}) respectively.
See also Corollary \ref{thm:cor of twisted Pieri}.

From the twisted Pieri formulas for Jack of Theorem\,\ref{thm:twisted Pieri}, we obtain three important results as follows.
The first one is an alternative proof of the following theorem on difference equations for interpolation Jack polynomials due to Knop--Sahi \cite{KS} (see also Corollary\,\ref{thm:main result 2 cor}).
\begin{Theorem}[Knop--Sahi]
\label{thm:Difference formula for IJP}
For any $\mathbf{x} \in \mathbb{C}^{r}$ and $\mathbf{k} \in \mathcal{P}$, we have
\begin{gather}
\label{eq:diff eq for IJP}
D_{r}^{(d)\,\mathrm{ip}}(u;\mathbf{x})P_{\mathbf{k}}^{\mathrm{ip}}
\bigg(\mathbf{x}+\frac{d}{2}\boldsymbol{\delta } ;\frac{d}{2}\bigg)
 = P_{\mathbf{k}}^{\mathrm{ip}}\bigg(\mathbf{x}+\frac{d}{2}\boldsymbol{\delta } ;\frac{d}{2}\bigg)I_{r}^{(d)}(u;\mathbf{k}),
\end{gather}
where
\begin{gather*}
D_{r}^{(d)\,\mathrm{ip}}(u;\mathbf{x}):= \sum_{\substack{J \subseteq [r]}}
 (-1)^{|J|} I_{J^{c}}^{(d)}(u;\mathbf{x}) A_{-,J}^{(d)}(\mathbf{x})
 \prod_{j \in J}\bigg(x_{j}+\frac{d}{2}(r-j)\bigg) T_{\mathbf{x}}^{J},
 \\
T_{x_{j}}f(\mathbf{x}):= f(\mathbf{x}-\epsilon _{j}), \qquad
T_{\mathbf{x}}^{J}:=\prod_{j \in J}T_{x_{j}}.
\end{gather*}
\end{Theorem}
The second result is the Pieri formulas for interpolation Jack polynomials (see also Corollary~\ref{thm:main result 1 cor}).
\begin{Theorem}
\label{thm:main result 1}
For any $\mathbf{x} \in \mathbb{C}^{r}$ and $\mathbf{k} \in \mathcal{P}$, we have
\begin{gather}
\label{eq:Pieri for IJP}
 I_{r}^{(d)}(u;\mathbf{x})
\frac{P_{\mathbf{k}}^{\mathrm{ip}}\left(\mathbf{x}+\frac{d}{2}\boldsymbol{\delta } ;\frac{d}{2}\right)}{P_{\mathbf{k}}\left(\mathbf{1} ;\frac{d}{2}\right)}
 =
 \sum_{\substack{J \subseteq [r], \\ \mathbf{k}+\epsilon_{J} \in \mathcal{P}}}
 \frac{P_{\mathbf{k}+\epsilon_{J}}^{\mathrm{ip}}\left(\mathbf{x}+\frac{d}{2}\boldsymbol{\delta } ;\frac{d}{2}\right)}{P_{\mathbf{k}+\epsilon_{J}}\left(\mathbf{1} ;\frac{d}{2}\right)}
 I_{J^{c}}^{(d)}\left(u;\mathbf{k}\right)A_{+,J}^{(d)}(\mathbf{k}).
\end{gather}
\end{Theorem}
Finally, we obtain the following intertwining relation for a kernel function of Jack polynomials~\cite{L,VK}
\begin{gather*}
{_{0}\mathcal{F}_0}^{(d)}\left( \mathbf{z},\mathbf{w}\right)
 := \sum_{\mathbf{m} \in \mathcal{P}}
 \Psi_{\mathbf{m}}^{(d)}(\mathbf{z})\Phi_{\mathbf{m}}^{(d)}(\mathbf{w})
 = \sum_{\mathbf{m} \in \mathcal{P}}
 \Phi_{\mathbf{m}}^{(d)}(\mathbf{z})\Psi_{\mathbf{m}}^{(d)}(\mathbf{w}),
\end{gather*}
which is a multivariate analogue of
\begin{gather*}
e^{zw} =
 \sum_{m=0}^{\infty}
 \frac{1}{m!}z^{m}w^{m}
 =
 \sum_{m=0}^{\infty}
 \Psi_{m}(z)\Phi_{m}(w)
 =
 \sum_{m=0}^{\infty}
 \Phi_{m}(z)\Psi_{m}(w).
\end{gather*}
\begin{Theorem}
\label{thm:intertwining rel}
For any $l=0,1,\ldots ,r$, we have
\begin{align}
\label{eq:0F0 kernel rel}
\left(\frac{d}{2}\right)^{l}
\left[\frac{(\ad{|\partial_{\mathbf{z}}|})^{l}}{l!}H_{r,l}^{(d)}(\mathbf{z})\right]
{_{0}\mathcal{F}_0}^{(d)}\left( \mathbf{z},\mathbf{w}\right)
 =
 {_{0}\mathcal{F}_0}^{(d)}\left( \mathbf{z},\mathbf{w}\right)e_{r,l}(\mathbf{w}).
\end{align}
\end{Theorem}
\noindent
It is a multivariate analogue of the relation
\begin{gather*}
\partial_{z}e^{zw}
 =
 e^{zw}w
\end{gather*}
and a higher order analogue of the formula in \cite{L} Section~14
\begin{gather*}
|\partial_{\mathbf{z}}|{_{0}\mathcal{F}_0}^{(d)}\left( \mathbf{z},\mathbf{w}\right)
 = {_{0}\mathcal{F}_0}^{(d)}\left( \mathbf{z},\mathbf{w}\right)|\mathbf{w}|.
\end{gather*}

The contents of this article are as follows.
In Section~\ref{section2}, we prove the twisted Pieri formulas for Jack polynomials.
From the twisted Pieri formulas for Jack polynomials, we give another proof of Theorem~\ref{thm:Difference formula for IJP} in Section~\ref{section3} and prove Theorem \ref{thm:main result 1} in Section~\ref{section4}.
We also prove the intertwining relation (\ref{eq:0F0 kernel rel}) for the kernel function ${_{0}\mathcal{F}_0}^{(d)}\left( \mathbf{z},\mathbf{w}\right)$ in Section~\ref{section5}.
Finally, we mention some future works for twisted Pieri formulas and their applications in Section~\ref{section6}.

\section{Twisted Pieri formulas for Jack polynomials}\label{section2}
To prove Theorem \ref{thm:twisted Pieri}, we need the following summation formula.
\begin{lemma}[mysterious summation]
For any $I \subseteq [r]$ and $\mathbf{x}=(x_{1},\ldots ,x_{r}) \in \mathcal{P}$, we have
\begin{gather}
\sum_{i \in I} \left(x_{i}+1+\frac{d}{2}(r-i)\right)
 A_{-,i,I\setminus {i}}^{(d)}(\mathbf{x}+\epsilon_{i})
 A_{+,i,I\setminus {i}}^{(d)}(\mathbf{x})\nonumber
 \\ \qquad
\label{eq:Mysterious sum}
{}-\sum_{i \in I} \left(x_{i}+\frac{d}{2}(r-i)\right)
 A_{+,i,I\setminus {i}}^{(d)}(\mathbf{x}-\epsilon_{i})
 A_{-,i,I\setminus {i}}^{(d)}(\mathbf{x}) =|I|,
\end{gather}
where $|I|$ is the cardinality of $I$ and
\begin{gather*}
A_{\pm ,i,I\setminus {i}}^{(d)}(\mathbf{x}) :=
 \prod_{j \in I\setminus {i}}\frac{x_{i}-x_{j}-\frac{d}{2}(i-j)\pm \frac{d}{2}}{x_{i}-x_{j}-\frac{d}{2}(i-j)}.
\end{gather*}
\end{lemma}
If $r=1$, then (\ref{eq:Mysterious sum}) is equal to a trivial summation
\begin{gather*}
(x+1)-x=1.
\end{gather*}
\begin{proof}
For convenience, we put
\begin{gather*}
s_{j}:=x_{j}+\frac{d}{2}(r-j).
\end{gather*}
By Pieri type formulas for the Jack polynomials (\ref{eq:PhiPieri-}) and (\ref{eq:PhiPieri+}), we have
\begin{gather*}
[|\partial_{\mathbf{z}}|,|\mathbf{z}|]\Phi_{\mathbf{x}}^{(d)}(\mathbf{z})
\\[-.3ex] \qquad
{} = \sum_{i=1}^{r}
 |\partial_{\mathbf{z}}|\Phi_{\mathbf{x}+\epsilon_{i}}^{(d)}(\mathbf{z})
 A_{+,i}^{(d)}(\mathbf{x}) - \sum_{j=1}^{r} |\mathbf{z}|
 \Phi_{\mathbf{x}-\epsilon_{j}}^{(d)}(\mathbf{z}) s_{j}
 A_{-,j}^{(d)}(\mathbf{x})
 \\[-.3ex] \qquad
{} = \sum_{i=1}^{r}
 \sum_{j=1}^{r}
 \Phi_{\mathbf{x}+\epsilon_{i}-\epsilon_{j}}^{(d)}(\mathbf{z})
\big((s_{j}+\delta _{j,i})
 A_{-,j}^{(d)}(\mathbf{x}+\epsilon_{i})
 A_{+,i}^{(d)}(\mathbf{x})
 -s_{j}A_{-,j}^{(d)}(\mathbf{x})
 A_{+,i}^{(d)}(\mathbf{x}-\epsilon_{j})\big)
 \\[-.3ex] \qquad
{} = \Phi_{\mathbf{x}}^{(d)}(\mathbf{z})
\sum_{i=1}^{r}\big((s_{i}+1)
 A_{-,i}^{(d)}(\mathbf{x}+\epsilon_{i})
 A_{+,i}^{(d)}(\mathbf{x})
 -s_{i}A_{-,i}^{(d)}(\mathbf{x})
 A_{+,i}^{(d)}(\mathbf{x}-\epsilon_{i})\big).
\end{gather*}
On the other hand,
\begin{gather*}
[|\partial_{\mathbf{z}}|,|\mathbf{z}|]\Phi_{\mathbf{x}}^{(d)}(\mathbf{z})
=\Phi_{\mathbf{x}}^{(d)}(\mathbf{z})r.
\end{gather*}
Then, we obtain (\ref{eq:Mysterious sum}).
\end{proof}

\begin{proof}[Proof of Theorem \ref{thm:twisted Pieri}]
Since (\ref{eq:Phikey lemma}) and (\ref{eq:Psikey lemma}) can be similarly proved, we only prove (\ref{eq:Phikey lemma}).
These formulas are proved by induction on $l$.

The case of $l=0$ is
\begin{gather*}
S_{r}^{(d)}(u;\mathbf{z})\Phi_{\mathbf{x}}^{(d)}(\mathbf{z}) =
\Phi_{\mathbf{x}}^{(d)}(\mathbf{z})I_{r}^{(d)}(u;\mathbf{x}).
\end{gather*}
This is (\ref{eq:Sekiguchi gen})
\begin{gather*}
S_{r}^{(d)}(u;\mathbf{z})P_{\mathbf{m}}\left(\mathbf{z};\frac{d}{2}\right)
 = P_{\mathbf{m}}\left(\mathbf{z};\frac{d}{2}\right) I_{r}^{(d)}(u;\mathbf{m})
\end{gather*}
exactly.

If $l=1$, then
\begin{gather*}
[(\ad{|\partial_{\mathbf{z}}|})S_{r}^{(d)}(u;\mathbf{z})]\Phi_{\mathbf{x}}^{(d)}(\mathbf{z})
\\[-.3ex] \qquad
{}= |\partial_{\mathbf{z}}|\Phi_{\mathbf{x}}^{(d)}(\mathbf{z})I_{r}^{(d)}(u;\mathbf{x})
 -S_{r}^{(d)}(u;\mathbf{z})\!\!
 \sum_{i=1}^{r}\Phi_{\mathbf{x}-\epsilon_{i}}^{(d)}(\mathbf{z})
 \!\left(x_{i}+\frac{d}{2}(r-i)\right)\!A_{-,i}^{(d)}(\mathbf{x})
 \\[-.3ex] \qquad
 {}= \sum_{i=1}^{r} \Phi_{\mathbf{x}-\epsilon_{i}}^{(d)}(\mathbf{z})
 s_{i}A_{-,i}^{(d)}(\mathbf{x})
 (I_{r}^{(d)}(u;\mathbf{x})-I_{r}^{(d)}(u;\mathbf{x}-\epsilon_{i}))
 \\[-.3ex] \qquad
 {}= \sum_{\substack{J \subseteq [r], \\[-.3ex] |J|=1}}
 \Phi_{\mathbf{x}-\epsilon_{J}}^{(d)}(\mathbf{z})
 I_{J^{c}}^{(d)}(u;\mathbf{x}) A_{-,J}^{(d)}(\mathbf{x})s_{i}.
\end{gather*}
Here, the first and second equalities follow from (\ref{eq:Sekiguchi gen}) and (\ref{eq:PhiPieri-}) respectively.

Assume the $n=l$ case holds.
Hence, from the induction hypothesis and (\ref{eq:PhiPieri-}) we have
\begin{gather*}
\left[\frac{(\ad{|\partial_{\mathbf{z}}|})^{l+1}}{(l+1)!}S_{r}^{(d)}(u;\mathbf{z})\right]
\Phi_{\mathbf{x}}^{(d)}(\mathbf{z})
\\[-.3ex] \qquad
{}=\frac{1}{l+1}|\partial_{\mathbf{z}}|\left[\frac{(\ad{|\partial_{\mathbf{z}}|})^{l}}{l!}
S_{r}^{(d)}(u;\mathbf{z})\right]\Phi_{\mathbf{x}}^{(d)}(\mathbf{z}) -\left[\frac{(\ad{|\partial_{\mathbf{z}}|})^{l}}{l!}S_{r}^{(d)}(u;\mathbf{z})\right]
\frac{1}{l+1}|\partial_{\mathbf{z}}|\Phi_{\mathbf{x}}^{(d)}(\mathbf{z})
\\[-.3ex] \qquad
{}= \frac{1}{l+1} \sum_{\substack{J \subseteq [r], \\[-.3ex] |J|=l}} \sum_{\nu =1}^{r}
 \bigg\{ \Phi_{\mathbf{x}-\epsilon_{J\sqcup \{\nu \}}}^{(d)}(\mathbf{z})
 A_{-,\nu }^{(d)}(\mathbf{x}-\epsilon_{J})
 (s_{\nu}-\chi _{J}(\nu ))
 I_{J^{c}}^{(d)}(u;\mathbf{x})
 A_{-,J}^{(d)}(\mathbf{x})
 \prod_{j \in J}s_{j}
\\[-.3ex] \qquad \qquad \qquad
{} - \Phi_{\mathbf{x}-\epsilon_{J\sqcup \{\nu \}}}^{(d)}(\mathbf{z})
 I_{J^{c}}^{(d)}(u;\mathbf{x}-\epsilon_{\nu })
 A_{-,J}^{(d)}(\mathbf{x}-\epsilon_{\nu}) A_{-,\nu }^{(d)}(\mathbf{x})
 s_{\nu }\prod_{j \in J}(s_{j}-\delta _{j,\nu})\bigg\},
\end{gather*}
where $\delta _{j,\nu}$ is the Kronecker's delta and
\begin{gather*}
\chi _{J}(\nu ) :=
 \begin{cases}
 1, & \nu \in J, \\[-.3ex]
 0, & \nu \not\in J.
 \end{cases}
\end{gather*}
From a simple calculation, we have
\begin{gather*}
 \left[\frac{(\ad{|\partial_{\mathbf{z}}|})^{l+1}}{(l+1)!}S_{r}^{(d)}(u;\mathbf{z})\right]
\Phi_{\mathbf{x}}^{(d)}(\mathbf{z})
= \sum_{\substack{I \subseteq [r], \\[-.3ex] |I|=l+1}} \frac{1}{l+1}
 \Phi_{\mathbf{x}-\epsilon_{I}}^{(d)}(\mathbf{z}) I_{I^{c}}^{(d)}(u;\mathbf{x})
 A_{-,I}^{(d)}(\mathbf{x}) \prod_{j \in I}s_{j}
 \\[-.3ex] \hphantom{ \bigg[\frac{(\ad{|\partial_{\mathbf{z}}|})^{l+1}}{(l+1)!}S_{r}^{(d)}(u;\mathbf{z})\bigg]
\Phi_{\mathbf{x}}^{(d)}(\mathbf{z})=}
 {}\times \sum_{i \in I} \bigg\{ \bigg(s_{i}+\frac{d}{2}u\bigg)
 A_{-,\{i\},I\setminus \{i\}}^{(d)}(\mathbf{x}+\epsilon_{i})
 A_{+,\{i\},I\setminus \{i\}}^{(d)}(\mathbf{x})
 \\[-.3ex] \hphantom{ \bigg[\frac{(\ad{|\partial_{\mathbf{z}}|})^{l+1}}{(l+1)!}S_{r}^{(d)}(u;\mathbf{z})\bigg]
\Phi_{\mathbf{x}}^{(d)}(\mathbf{z})=}
 {}-\bigg(s_{i}-1+\frac{d}{2}u\bigg) A_{-,\{i\},I\setminus \{i\}}^{(d)}(\mathbf{x})
 A_{+,\{i\},I\setminus \{i\}}^{(d)}(\mathbf{x}-\epsilon_{i}) \bigg\}.
\end{gather*}
Since the summation
\begin{gather*}
\sum_{i \in I}
 \big\{s_{i}A_{-,\{i\},I\setminus \{i\}}^{(d)}(\mathbf{x}+\epsilon_{i})A_{+,\{i\},I\setminus \{i\}}^{(d)}(\mathbf{x})
-(s_{i}-1)A_{-,\{i\},I\setminus \{i\}}^{(d)}(\mathbf{x})A_{+,\{i\},I\setminus \{i\}}^{(d)}(\mathbf{x}-\epsilon_{i})\big\}
\\[-.3ex] \qquad
{}=l+1
\end{gather*}
is our mysterious summation (\ref{eq:Mysterious sum}) exactly, we obtain the conclusion (\ref{eq:Phikey lemma}).
\end{proof}

\section[Another proof of difference equations for interpolation Jack polynomials]{Another proof of difference equations\\ for interpolation Jack polynomials}\label{section3}

\begin{proof}[Proof of Theorem \ref{thm:Difference formula for IJP}]
Since the difference equation (\ref{eq:diff eq for IJP}) is a relation for rational function of $(x_{1}, \ldots ,x_{r})$, it is enough to prove (\ref{eq:diff eq for IJP}) for any partition $\mathbf{x} \in \mathcal{P}$.
To prove (\ref{eq:diff eq for IJP}), we compute
\begin{gather*}
S_{r}^{(d)}(u;\mathbf{z})\Phi_{\mathbf{x}}^{(d)}(\mathbf{1}+\mathbf{z})
\end{gather*}
in two different ways.
First, a simple calculation shows that
\begin{gather*}
S_{r}^{(d)}(u;\mathbf{z})\Phi_{\mathbf{x}}^{(d)}(\mathbf{1}+\mathbf{z})
 = S_{r}^{(d)}(u;\mathbf{z})e^{|\partial_{\mathbf{z}}|}\Phi_{\mathbf{x}}^{(d)}(\mathbf{z})
 =e^{|\partial_{\mathbf{z}}|}e^{-|\partial_{\mathbf{z}}|}S_{r}^{(d)}(u;\mathbf{z})
 e^{|\partial_{\mathbf{z}}|}\Phi_{\mathbf{x}}^{(d)}(\mathbf{z})
 \\ \hphantom{S_{r}^{(d)}(u;\mathbf{z})\Phi_{\mathbf{x}}^{(d)}(\mathbf{1}+\mathbf{z})}
{} = e^{|\partial_{\mathbf{z}}|}\big[e^{-\ad{|\partial_{\mathbf{z}}|}}S_{r}^{(d)}(u;\mathbf{z})\big]
\Phi_{\mathbf{x}}^{(d)}(\mathbf{z}).
\end{gather*}
Since the highest derivative in $H_{r,p}^{(d)}(\mathbf{z})$ has degree $p$, the sum
\begin{gather*}
e^{-\ad{|\partial_{\mathbf{z}}|}}S_{r}^{(d)}(u;\mathbf{z})
 = \sum_{l\geq 0}
 \frac{(-\ad{|\partial_{\mathbf{z}}|})^{l}}{l!}S_{r}^{(d)}(u;\mathbf{z})
\end{gather*}
terminates after $(-\ad{|\partial_{\mathbf{z}}|})^{r}$.
Then, we have
\begin{gather*}
S_{r}^{(d)}(u;\mathbf{z})\Phi_{\mathbf{x}}^{(d)}(\mathbf{1}+\mathbf{z})
 = e^{|\partial_{\mathbf{z}}|} \sum_{l=0}^{r}
 \left[\frac{(-\ad{|\partial_{\mathbf{z}}|})^{l}}{l!}S_{r}^{(d)}(u;\mathbf{z})\right]
 \Phi_{\mathbf{x}}^{(d)}(\mathbf{z}).
\end{gather*}
By applying the twisted Pieri (\ref{eq:Phikey lemma}) and the binomial (\ref{eq:binomial}), we have
\begin{gather}
 S_{r}^{(d)}(u;\mathbf{z})\Phi_{\mathbf{x}}^{(d)}(\mathbf{1}+\mathbf{z})
 = \sum_{l=0}^{r}(-1)^{l} \sum_{\substack{J \subseteq [r], \\ |J|=l}}
 e^{|\partial_{\mathbf{z}}|}\Phi_{\mathbf{x}-\epsilon_{J}}^{(d)}(\mathbf{z})
 I_{J^{c}}^{(d)}(u;\mathbf{x}) A_{-,J}^{(d)}(\mathbf{x})
 \prod_{j \in J}\left(x_{j}+\frac{d}{2}(r-j)\right) \nonumber
 \\ \hphantom{S_{r}^{(d)}(u;\mathbf{z})\Phi_{\mathbf{x}}^{(d)}(\mathbf{1}+\mathbf{z})}
 {}= \sum_{J \subseteq [r]} (-1)^{|J|}
 \Phi_{\mathbf{x}-\epsilon_{J}}^{(d)}(\mathbf{1}+\mathbf{z})
 I_{J^{c}}^{(d)}(u;\mathbf{x}) A_{-,J}^{(d)}(\mathbf{x})
 \prod_{j \in J}\left(x_{j}+\frac{d}{2}(r-j)\right) \nonumber
 \\ \hphantom{S_{r}^{(d)}(u;\mathbf{z})\Phi_{\mathbf{x}}^{(d)}(\mathbf{1}+\mathbf{z})}
 {} = \sum_{\mathbf{k} \in \mathcal{P}} \Psi_{\mathbf{k}}^{(d)}(\mathbf{z})\!\!\!
 \sum_{J \subseteq [r]} (-1)^{|J|}
 \frac{P_{\mathbf{k}}^{\mathrm{ip}}\left(\mathbf{x}-\epsilon_{J}+\frac{d}{2}\boldsymbol{\delta } ;\frac{d}{2}\right)}{P_{\mathbf{k}}\left(\mathbf{1};\frac{d}{2}\right)}
 I_{J^{c}}^{(d)}(u;\mathbf{x}) \nonumber
 \\ \hphantom{S_{r}^{(d)}(u;\mathbf{z})\Phi_{\mathbf{x}}^{(d)}(\mathbf{1}+\mathbf{z})=}
\label{eq:difference eq lhs}
{} \times A_{-,J}^{(d)}(\mathbf{x}) \prod_{j \in J}\bigg(x_{j}+\frac{d}{2}(r-j)\bigg).
\end{gather}

On the other hand, from the binomial formula (\ref{eq:binomial}) and (\ref{eq:Sekiguchi gen}), we have
\begin{gather}
S_{r}^{(d)}(u;\mathbf{z})\Phi_{\mathbf{x}}^{(d)}(\mathbf{1}+\mathbf{z})
= \sum_{\mathbf{k} \subseteq \mathbf{x}}
\frac{P_{\mathbf{k}}^{\mathrm{ip}}\left(\mathbf{x}+\frac{d}{2}\boldsymbol{\delta } ;\frac{d}{2}\right)}{P_{\mathbf{k}}\left(\mathbf{1};\frac{d}{2}\right)}
 S_{r}^{(d)}(u;\mathbf{z})\Psi_{\mathbf{k}}^{(d)}(\mathbf{z})
\nonumber\\ \hphantom{S_{r}^{(d)}(u;\mathbf{z})\Phi_{\mathbf{x}}^{(d)}(\mathbf{1}+\mathbf{z})}
{}= \sum_{\mathbf{k} \subseteq \mathbf{x}}
 \frac{P_{\mathbf{k}}^{\mathrm{ip}}\left(\mathbf{x}+\frac{d}{2}\boldsymbol{\delta } ;\frac{d}{2}\right)}{P_{\mathbf{k}}\left(\mathbf{1};\frac{d}{2}\right)}
 \Psi_{\mathbf{k}}^{(d)}(\mathbf{z})I_{r}^{(d)}(u;\mathbf{x})
\nonumber\\ \hphantom{S_{r}^{(d)}(u;\mathbf{z})\Phi_{\mathbf{x}}^{(d)}(\mathbf{1}+\mathbf{z})}
\label{eq:difference eq rhs}
{}= \sum_{\mathbf{k} \in \mathcal{P}} \Psi_{\mathbf{k}}^{(d)}(\mathbf{z})
 I_{r}^{(d)}(u;\mathbf{x})
 \frac{P_{\mathbf{k}}^{\mathrm{ip}}\left(\mathbf{x}+\frac{d}{2}\boldsymbol{\delta } ;\frac{d}{2}\right)}{P_{\mathbf{k}}\left(\mathbf{1};\frac{d}{2}\right)}.
\end{gather}
By comparing coefficients for $\Psi_{\mathbf{k}}^{(d)}(\mathbf{z})$ in (\ref{eq:difference eq lhs}) and (\ref{eq:difference eq rhs}), we obtain the conclusion (\ref{eq:diff eq for IJP}).
\end{proof}

Comparing coefficients for $u^{r-l}$ in (\ref{eq:diff eq for IJP}), we obtain higher order difference formulas for interpolation Jack polynomials.
\begin{Corollary}
\label{thm:main result 2 cor}
For any $\mathbf{x} \in \mathbb{C}^{r}$, $\mathbf{k} \in \mathcal{P}$ and $l=0,1,\ldots, r$, we have
\begin{gather*}
 e_{r,l}\bigg(\mathbf{k}+\frac{d}{2}\boldsymbol{\delta } \bigg)
 P_{\mathbf{k}}^{\mathrm{ip}}\bigg(\mathbf{x}+\frac{d}{2}\boldsymbol{\delta } ;\frac{d}{2}\bigg)
 =\!\!\!
 \sum_{\substack{J \subseteq [r], \\ 0\leq |J|\leq l}}\!\!\!
 (-1)^{|J|}P_{\mathbf{k}}^{\mathrm{ip}}\bigg(\mathbf{x}-\epsilon_{J}+\frac{d}{2}\boldsymbol{\delta } ;\frac{d}{2}\bigg)
 \\ \hphantom{ e_{r,l}\bigg(\mathbf{k}+\frac{d}{2}\boldsymbol{\delta } \bigg)
 P_{\mathbf{k}}^{\mathrm{ip}}\bigg(\mathbf{x}+\frac{d}{2}\boldsymbol{\delta } ;\frac{d}{2}\bigg)
 =}
{}\times e_{r-|J|,l-|J|}\bigg(\!\bigg(\mathbf{x}+\frac{d}{2}\boldsymbol{\delta }\bigg)_{\!J^{c}}\!\bigg)
 A_{-,J}^{(d)}(\mathbf{x})
 \prod_{j \in J}\!\bigg(x_{j}+\frac{d}{2}(r-j)\bigg),
\end{gather*}
where
\begin{gather*}
\bigg(\mathbf{x}+\frac{d}{2}\boldsymbol{\delta } \bigg)_{J^{c}}
 :=
 \bigg(x_{i_{1}}+\frac{d}{2}(r-i_{1}), \ldots, x_{i_{r-l}}+\frac{d}{2}(r-i_{r-l})\bigg)_{i_{1},\ldots,i_{r-l} \in J^{c}}.
\end{gather*}
\end{Corollary}
Originally, Theorem\,\ref{thm:Difference formula for IJP} or Corollary\,\ref{thm:main result 2 cor} were proved by Knop--Sahi \cite{KS}.
Knop--Sahi's proof shows that $D_{r}^{(d)\,\mathrm{ip}}(u;\mathbf{x})P_{\mathbf{k}}^{\mathrm{ip}}\left(\mathbf{x}+\frac{d}{2}\boldsymbol{\delta } ;\frac{d}{2}\right)$ satisfy the conditions ${\rm{(1)}{}^{\mathrm{ip}}}$ and ${\rm{(2)}{}^{\mathrm{ip}}}$ up to a constant $c(\mathbf{k})$ for any $\mathbf{k} \in \mathcal{P}$, and determine $c(\mathbf{k})(=I_{r}^{(d)}(u;\mathbf{k}))$ explicitly.
Knop--Sahi's proof requires that the explicit expression of the difference operator $D_{r}^{(d)\,\mathrm{ip}}(u;\mathbf{x})$ for interpolation Jack polynomials are known in ad hoc, whereas our proof does not require it.

\section{Pieri formulas for interpolation Jack polynomials}\label{section4}

\begin{proof}[Proof of Theorem \ref{thm:main result 1}]
As with the proof of Theorem \ref{thm:Difference formula for IJP}, it is enough to prove (\ref{eq:Pieri for IJP}) for $\mathbf{x} \in \mathcal{P}$.
For the purpose, we compute
\begin{gather*}
[e^{\ad{|\partial_{\mathbf{z}}|}}S_{r}^{(d)}(u;\mathbf{z})]\Phi_{\mathbf{x}}^{(d)}(\mathbf{1}+\mathbf{z})
 = e^{|\partial_{\mathbf{z}}|}S_{r}^{(d)}(u;\mathbf{z})e^{-|\partial_{\mathbf{z}}|}
 \Phi_{\mathbf{x}}^{(d)}(\mathbf{1}+\mathbf{z})
\end{gather*}
in two different ways.
From the binomial (\ref{eq:binomial}) and the twisted Pieri (\ref{eq:Phikey lemma})
\begin{gather*}
[e^{\ad{|\partial_{\mathbf{z}}|}}S_{r}^{(d)}(u;\mathbf{z})]\Phi_{\mathbf{x}}^{(d)}(\mathbf{1}+\mathbf{z})
 \\ \qquad
 {} = \sum_{\mathbf{k} \subseteq \mathbf{x}} \frac{P_{\mathbf{k}}^{\mathrm{ip}}
\left(\mathbf{x}+\frac{d}{2}\boldsymbol{\delta } ;\frac{d}{2}\right)}{P_{\mathbf{k}}
\left(\mathbf{1};\frac{d}{2}\right)} \sum_{l=0}^{r}
 \left[\frac{(\ad{|\partial_{\mathbf{z}}|})^{l}}{l!}S_{r}^{(d)}(u;\mathbf{z})\right]
 \Psi_{\mathbf{k}}^{(d)}(\mathbf{z})
 \\ \qquad
 {}= \sum_{\mathbf{k} \subseteq \mathbf{x}}
 \frac{P_{\mathbf{k}}^{\mathrm{ip}}\left(\mathbf{x}+\frac{d}{2}\boldsymbol{\delta } ;\frac{d}{2}\right)}{P_{\mathbf{k}}\left(\mathbf{1};\frac{d}{2}\right)} \sum_{l=0}^{r}
 \sum_{\substack{J \subseteq [r], |J|=l, \\ \mathbf{k}-\epsilon_{J} \in \mathcal{P}}}
 \Psi_{\mathbf{k}-\epsilon_{J}}^{(d)}(\mathbf{z}) I_{J^{c}}^{(d)}(u;\mathbf{k})
 h_{+,J}^{(d)}(\mathbf{k}-\epsilon_{J})
 \\ \qquad
 {}= \sum_{\mathbf{k} \in \mathcal{P}} \Psi_{\mathbf{k}}^{(d)}(\mathbf{z})
 \sum_{\substack{J \subseteq [r], \\ \mathbf{k}+\epsilon_{J} \in \mathcal{P}}}
 \frac{P_{\mathbf{k}}^{\mathrm{ip}}\left(\mathbf{x}+\frac{d}{2}\boldsymbol{\delta } ;\frac{d}{2}\right)}{P_{\mathbf{k}+\epsilon_{J}}\left(\mathbf{1};\frac{d}{2}\right)}
 I_{J^{c}}^{(d)}\left(u;\mathbf{k}\right)
 A_{+,J}^{(d)}(\mathbf{k}).
\end{gather*}
The third equality follows from $I_{J^{c}}^{(d)}\left(u;\mathbf{k}+\epsilon_{J}\right)=I_{J^{c}}^{(d)}(u;\mathbf{k})$.

On the other hand, from (\ref{eq:Sekiguchi gen}) and the binomial formula (\ref{eq:binomial}),
\begin{gather*}
 [e^{\ad{|\partial_{\mathbf{z}}|}}S_{r}^{(d)}(u;\mathbf{z})]
 \Phi_{\mathbf{x}}^{(d)}(\mathbf{1}+\mathbf{z}) = e^{|\partial_{\mathbf{z}}|}S_{r}^{(d)}(u;\mathbf{z})e^{-|\partial_{\mathbf{z}}|}
 \Phi_{\mathbf{x}}^{(d)}(\mathbf{1}+\mathbf{z})
 = e^{|\partial_{\mathbf{z}}|}S_{r}^{(d)}(u;\mathbf{z})\Phi_{\mathbf{x}}^{(d)}(\mathbf{z})
 \\ \hphantom{[e^{\ad{|\partial_{\mathbf{z}}|}}S_{r}^{(d)}(u;\mathbf{z})]
 \Phi_{\mathbf{x}}^{(d)}(\mathbf{1}+\mathbf{z})}
 = e^{|\partial_{\mathbf{z}}|}\Phi_{\mathbf{x}}^{(d)}(\mathbf{z})I_{r}^{(d)}(u;\mathbf{x})
 = \Phi_{\mathbf{x}}^{(d)}(\mathbf{1}+\mathbf{z})I_{r}^{(d)}(u;\mathbf{x})
\\ \hphantom{[e^{\ad{|\partial_{\mathbf{z}}|}}S_{r}^{(d)}(u;\mathbf{z})]
 \Phi_{\mathbf{x}}^{(d)}(\mathbf{1}+\mathbf{z})}
{} =
 \sum_{\mathbf{k} \subseteq \mathbf{x}}
 \frac{P_{\mathbf{k}}^{\mathrm{ip}}\left(\mathbf{x}+\frac{d}{2}\boldsymbol{\delta } ;\frac{d}{2}\right)}{P_{\mathbf{k}}\left(\mathbf{1};\frac{d}{2}\right)}
 \Psi_{\mathbf{k}}^{(d)}(\mathbf{z})I_{r}^{(d)}(u;\mathbf{x})
 \\ \hphantom{[e^{\ad{|\partial_{\mathbf{z}}|}}S_{r}^{(d)}(u;\mathbf{z})]
 \Phi_{\mathbf{x}}^{(d)}(\mathbf{1}+\mathbf{z})}
 {} = \sum_{\mathbf{k} \in \mathcal{P}} \Psi_{\mathbf{k}}^{(d)}(\mathbf{z})
 I_{r}^{(d)}(u;\mathbf{x})
 \frac{P_{\mathbf{k}}^{\mathrm{ip}}\left(\mathbf{x}+\frac{d}{2}\boldsymbol{\delta } ;\frac{d}{2}\right)}{P_{\mathbf{k}}\left(\mathbf{1};\frac{d}{2}\right)}.\tag*{\qed}
\end{gather*}\renewcommand{\qed}{}
\end{proof}

By comparing coefficients for $u^{r-l}$ in (\ref{eq:Pieri for IJP}), we obtain the Pieri type formulas for the inter\-polation Jack polynomials, which are a higher order analogue of equation~(5.3) in \cite{OO1} or equation~(14.2) in~\cite{L}.
\begin{Corollary}\label{thm:main result 1 cor}
For any $\mathbf{x} \in \mathbb{C}^{r}$, $\mathbf{k} \in \mathcal{P}$ and $l=0,1,\ldots, r$,
\begin{gather*}
 e_{r,l}\left(\mathbf{x}+\frac{d}{2}\boldsymbol{\delta } \right)\frac{P_{\mathbf{k}}^{\mathrm{ip}}\left(\mathbf{x}+\frac{d}{2}\boldsymbol{\delta } ;\frac{d}{2}\right)}{P_{\mathbf{k}}\left(\mathbf{1} ;\frac{d}{2}\right)} \nonumber
 \\ \qquad
\label{eq:Pieri for IJP 2}
{}= \sum_{\substack{J \subseteq [r],\, |J|=l, \\ \mathbf{k}+\epsilon_{J} \in \mathcal{P}}}
 \frac{P_{\mathbf{k}+\epsilon_{J}}^{\mathrm{ip}}\left(\mathbf{x}+\frac{d}{2}\boldsymbol{\delta } ;\frac{d}{2}\right)}{P_{\mathbf{k}+\epsilon_{J}}\left(\mathbf{1} ;\frac{d}{2}\right)}
 e_{r-|J|,l-|J|}\left(\left(\mathbf{k}+\frac{d}{2}\boldsymbol{\delta } \right)_{J^{c}}\right)
 A_{+,J}^{(d)}(\mathbf{k}).
\end{gather*}
\end{Corollary}

\section[Some intertwining relations for the kernel function 0F0(d)(z,w)]
{Some intertwining relations for the kernel function ${_{\mathbf{0}}\mathcal{F}_{\mathbf{0}}}^{\boldsymbol{(d)}}\boldsymbol{(}\mathbf{z},\mathbf{w}\boldsymbol{)}$}\label{section5}

By comparing the coefficients for $u^{r-l}$ of the twisted Pieri formulas (\ref{eq:Phikey lemma}) and (\ref{eq:Psikey lemma}),
we obtain the following twisted Pieri type formulas.
\begin{Corollary}
\label{thm:cor of twisted Pieri}
For any $\mathbf{z} \in \mathbb{C}^{r}$ and $l=0,1,\ldots, r$, we have
\begin{gather}
\label{eq:Phikey lemma 2}
\left(\frac{d}{2}\right)^{l}
\left[\frac{(\ad{|\partial_{\mathbf{z}}|})^{l}}{l!}H_{r,l}^{(d)}(\mathbf{z})\right]
\Phi_{\mathbf{x}}^{(d)}(\mathbf{z})
 = \!\!\!\sum_{\substack{J \subseteq [r], |J|=l, \\ \mathbf{x}-\epsilon_{J} \in \mathcal{P}}}\!\!\!
 \Phi_{\mathbf{x}-\epsilon_{J}}^{(d)}(\mathbf{z})
 A_{-,J}^{(d)}(\mathbf{x}) \prod_{j \in J}\bigg(x_{j}+\frac{d}{2}(r-j)\bigg),
 \\
\label{eq:Psikey lemma 2}
\left(\frac{d}{2}\right)^{l}
\left[\frac{(\ad{|\partial_{\mathbf{z}}|})^{l}}{l!}H_{r,l}^{(d)}(\mathbf{z})\right]
\Psi_{\mathbf{x}}^{(d)}(\mathbf{z})
 =\!\!\! \sum_{\substack{J \subseteq [r], |J|=l, \\ \mathbf{x}-\epsilon_{J} \in \mathcal{P}}}\!\!\!
 \Psi_{\mathbf{x}-\epsilon_{J}}^{(d)}(\mathbf{z})
 A_{+,J}^{(d)}(\mathbf{x}-\epsilon_{J}).
\end{gather}
\end{Corollary}
On the other hand, by comparing the coefficients for the highest degree term of (\ref{eq:Pieri for IJP 2})
and the condition {\rm{(2)}${}^{\mathrm{ip}}$},
we obtain the following well-known Pieri formulas for the ordinary Jack polynomials.
\begin{Corollary}[Macdonald \cite{M2}, Stanley \cite{St}]
\label{thm:Jack Pieri}
For $l=0,1,\ldots, r$,
\begin{gather}
\label{eq:Jack Pieri}
 e_{r,l}(\mathbf{z}) \Phi_{\mathbf{k}}^{(d)}(\mathbf{z})
 = \sum_{\substack{J \subseteq [r], |J|=l, \\ \mathbf{k}+\epsilon_{J} \in \mathcal{P}}}
 \Phi_{\mathbf{k}+\epsilon_{J}}^{(d)}(\mathbf{z}) A_{+,J}^{(d)}(\mathbf{k}).
\end{gather}
\end{Corollary}
From two Pieri type formulas (\ref{eq:Psikey lemma 2}) and (\ref{eq:Jack Pieri}), we prove Theorem \ref{thm:intertwining rel} immediately.
In fact, we have
\begin{gather*}
 \left(\frac{d}{2}\right)^{l}
\left[\frac{(\ad{|\partial_{\mathbf{z}}|})^{l}}{l!}H_{r,l}^{(d)}(\mathbf{z})\right]
{_{0}\mathcal{F}_0}^{(d)}\left(\mathbf{z},\mathbf{w}\right) =
 \sum_{\mathbf{m} \in \mathcal{P}} \left(\frac{d}{2}\right)^{l}
\left[\frac{(\ad{|\partial_{\mathbf{z}}|})^{l}}{l!}H_{r,l}^{(d)}(\mathbf{z})\right]
 \Psi_{\mathbf{m}}^{(d)}(\mathbf{z})\Phi_{\mathbf{m}}^{(d)}(\mathbf{w})
 \\ \phantom{ \left(\frac{d}{2}\right)^{l}
\left[\frac{(\ad{|\partial_{\mathbf{z}}|})^{l}}{l!}H_{r,l}^{(d)}(\mathbf{z})\right]
{_{0}\mathcal{F}_0}^{(d)}\left(\mathbf{z},\mathbf{w}\right)}
{}= \sum_{\mathbf{m} \in \mathcal{P}}
 \sum_{\substack{J \subseteq [r], |J|=l, \\ \mathbf{m}-\epsilon_{J} \in \mathcal{P}}}
 \Psi_{\mathbf{m}-\epsilon_{J}}^{(d)}(\mathbf{z})
 A_{+,J}^{(d)}(\mathbf{m}-\epsilon_{J}) \Phi_{\mathbf{m}}^{(d)}(\mathbf{w})
 \\ \phantom{ \left(\frac{d}{2}\right)^{l}
\left[\frac{(\ad{|\partial_{\mathbf{z}}|})^{l}}{l!}H_{r,l}^{(d)}(\mathbf{z})\right]
{_{0}\mathcal{F}_0}^{(d)}\left(\mathbf{z},\mathbf{w}\right)}
{} = \sum_{\mathbf{m} \in \mathcal{P}} \Psi_{\mathbf{m}}^{(d)}(\mathbf{z})
 \sum_{\substack{J \subseteq [r], |J|=l, \\ \mathbf{m}+\epsilon_{J} \in \mathcal{P}}}
 \Phi_{\mathbf{m}+\epsilon_{J}}^{(d)}(\mathbf{w})A_{+,J}^{(d)}(\mathbf{m})
 \\ \phantom{ \left(\frac{d}{2}\right)^{l}
\left[\frac{(\ad{|\partial_{\mathbf{z}}|})^{l}}{l!}H_{r,l}^{(d)}(\mathbf{z})\right]
{_{0}\mathcal{F}_0}^{(d)}\left(\mathbf{z},\mathbf{w}\right)}
{} =\! \sum_{\mathbf{m} \in \mathcal{P}}\!\! \Psi_{\mathbf{m}}^{(d)}(\mathbf{z})
 e_{r,l}(\mathbf{w})\Phi_{\mathbf{m}}^{(d)}(\mathbf{w}) =
 {_{0}\mathcal{F}_0}^{(d)}\left( \mathbf{z},\mathbf{w}\right)e_{r,l}(\mathbf{w}).
\end{gather*}

\section{Concluding remarks}\label{section6}

In this article, we have demonstrated usefulness of the twisted Pieri formulas (\ref{eq:Phikey lemma}) and (\ref{eq:Psikey lemma}).
In fact, from (\ref{eq:Phikey lemma}) and (\ref{eq:Psikey lemma}), difference equations and Pieri formulas for interpolation Jack polynomials are derived in parallel.
Further, we also obtain some intertwining relations for ${_{0}\mathcal{F}_0}^{(d)}\left( \mathbf{z},\mathbf{w}\right)$, which means that the twisted Pieri formulas (\ref{eq:Phikey lemma 2}) and (\ref{eq:Psikey lemma 2}) are dual of the usual Pieri formulas (\ref{eq:Jack Pieri}).

All our results in this paper are of formulas associated with $A$-type root system.
Therefore it would be interesting to generalize twisted Pieri formulas from type $A$ to type $BC$.
Some extension to $q,t$-analogue of our theorems is also an important problem.
Sahi \cite{Sa2} introduced interpolation Macdonald polynomials which are $q, t$-analogue of interpolation Jack polynomials,
and Okounkov~\cite{O1} gave ``Idea of Proof'' of difference equations for interpolation Macdonald polynomials.
Another proof of difference equations for interpolation Macdonald polynomials in our approach would be desirable.
It would also be a challenging problem to explore a new method which is applicable further to the elliptic case~\cite{CG, R}.

\subsection*{Acknowledgements}
We are grateful to Professor Masatoshi Noumi (Kobe University) for his helpful advice on our paper.
We also wish to thank Professor Farrokh Atai for his valuable suggestions on Jack and interpolation Jack polynomials.
Further, we thank the referees for their helpful comments.
This work was supported by Grant-in-Aid for JSPS Fellows (Number 18J00233).

\pdfbookmark[1]{References}{ref}
\LastPageEnding

\end{document}